\documentclass[12pt,a4paper]{amsart}
\usepackage{amssymb}
\usepackage{verbatim}
\usepackage{graphicx}
\newtheorem{theorem}{Theorem}[section]
\newtheorem{lemma}[theorem]{Lemma}
\newtheorem{proposition}[theorem]{Proposition}

\newtheorem{fact}[theorem]{Fact}

\theoremstyle{definition}

\newtheorem{example}[theorem]{Example}

\theoremstyle{remark}
\newtheorem{remark}[theorem]{Remark}

\numberwithin{equation}{section}

\renewcommand{\dim}{\mathop{\mathrm{dim}}}

\newcommand{\R}{\mathbb{R}}

\newcommand{\X}{\mathrm{X}}
\renewcommand{\H}{\mathrm{H}}

\newcommand{\B}{\mathbf{B}}
\newcommand{\U}{\mathbf{U}}

\renewcommand{\S}{\mathbf{S}}

\begin{document}

\title[Convexity properties of quasihyperbolic balls]{Convexity properties of quasihyperbolic balls on Banach spaces}

\begin{abstract}
We study the convexity and starlikeness of  metric balls on Banach spaces when the metric is the quasihyperbolic metric or the distance ratio metric. In particular, problems related to these metrics on convex domains, and on punctured Banach spaces, are considered. 
\end{abstract}

\author{Antti Rasila}
\address{Antti Rasila, Aalto University, Institute of Mathematics, P.O. Box 11100, FI-00076 Aalto, Finland} 
\email{antti.rasila@iki.fi}

\author{Jarno Talponen}
\address{Jarno Talponen, Aalto University, Institute of Mathematics, P.O. Box 11100, FI-00076 Aalto, Finland} 
\email{talponen@cc.hut.fi}
\date{\today}
\maketitle

\section{Introduction}

In this paper we deal with Banach manifolds, which are obtained by defining a conformal metric on non-trivial subdomains of a given Banach space. An example of such metric is the quasihyperbolic metric on a domain of a Banach space. 
It is obtained from the norm-induced metric by adding a weight, which depends only on distance to the boundary of the domain. The quasihyperbolic metric of domains in ${\R}^n$ was first studied by F.W.~Gehring and his students B.~Palka \cite{GehringPalka76} and 
B.~Osgood \cite{GehringOsgood79} in 1970's. It has turned out to be a useful tool in, e.g., the theory of quasiconformal mappings. In particular, quasihyperbolic metric plays a crucial role in the theory of quasiconformal mappings in Banach spaces, developed by 
J.~V\"ais\"al\"a in the series of articles \cite{Vaisala90,Vaisala91,Vaisala92,Vaisala98,Vaisala99}. This is due to the fact that many of the tools used in the Euclidean space are not available in the infinite-dimensional setting (see \cite{Vaisala99}). 

We mainly study the question of how the geometry of the Banach space norm translates into the properties of the quasihyperbolic metric.  In particular, we consider convexity and starlikeness of quasihyperbolic balls in domains of Banach spaces, for example 
the punctured space $\Omega = \X \setminus \{0\}$.  This problem was posed in  ${\R}^n$ by M.~Vuorinen \cite{Vuorinen07}, and studied by R.~Kl\'en in \cite{Klen08a, Klen08b} and J. V\"{a}is\"{a}l\"{a} in \cite{Vaisala08}. 
Some of the techniques used there are specific to ${\R}^n$. In the general Banach space setting a very different approach is required. 
 
Our main results are the following. In Theorem \ref{thm: sl} we show that each ball in the distance ratio metric (the $j$-metric) defined on a proper subdomain of a Banach space is starlike for radii $r\le \log 2$, partly generalizing 
a result of Kl\'en \cite[Theorem 3.1]{Klen08b}. In Theorem \ref{thm: cd}, which improves a result of O.~Martio and J.~V\"ais\"al\"a \cite[2.13]{MartioVaisala11}, we show that the $j$-balls and the quasihyperbolic balls defined on a convex domain of 
a Banach space are convex. Then, in Theorem \ref{thm_starlike_starlike}, we show that all $j$-balls and quasihyperbolic balls are starlike if the domain is starlike with respect to the center of the ball.
We also give a counterexample, which settles a question posed by O.~Martio and J.~V\"ais\"al\"a \cite[2.14]{MartioVaisala11} concerning quasihyperbolic geodesics on uniformly convex Banach spaces. Related problems involving quasihyperbolic geodesics have been 
studied in ${\R}^n$ by G.J.~Martin \cite{Martin85} in the 1980's, and several authors thereafter. Finally, we discuss the issue of convexity of quasihyperbolic balls on punctured Banach spaces.

\section{Preliminaries}

First, let us recall a few basic results and definitions. Unless otherwise stated, we will assume that $\X$ is a Banach space with $\dim \X\ge 2$, and that $\Omega \subsetneq \X$ is a domain. 
Open and closed balls in $\X$ are
\[
\U(x,r) := \U_{\|\cdot\|}(x,r) := \{ y\in  \X : \|x-y\| <r\}
\]
\[
\B(x,r) :=\B_{\|\cdot\|}(x,r) := \{ y\in  \X : \|x-y\| \le r\},\text{ and } \S(x,r):=\partial \B(x,r). 
\]
 A set $\Omega\subset \X$ is called \emph{convex} if the line segment
\[
[x,y] := \{ tx + (1-t)y : t\in [0,1]\} \subset \Omega\textrm{ for all }x,y\in\Omega,
\] 
and \emph{starlike} with respect to $x_0\in \Omega$ if 
\[
[x_0,y] := \{ tx_0 + (1-t)y : t\in [0,1]\} \subset \Omega\textrm{ for all }y\in\Omega.
\] 
Observe that the use of notation $[x,y]$ here is different from some texts dealing with Banach spaces. Obviously a set $\Omega$ is convex if and only if it is starlike with respect to every point $x_0\in \Omega$.

\subsection{Paths and line integrals}

In what follows a {\it path} in a metric space $(\X,d)$ is a continuous mapping $\gamma$
of the unit interval $I=[0,1]$ into $\X$. If $J=[a,b]\subset I$ is a closed
subinterval, then the {\it length} of a path $\gamma\colon I\to \X$ restricted to $J$ is
\begin{equation}
\label{lendef}
\ell_{d}(\gamma,a,b) = \sup \sum_{i=1}^n d\big(\gamma(t_i),\gamma(t_{i+1})\big),
\end{equation}
where the supremum is taken over all sequences $a=t_1\leq t_2 \leq
\ldots \leq t_n\leq t_{n+1}=b$. The (total) length of $\gamma$ is $\ell_{d}(\gamma)=\ell_{d}(\gamma,0,1)$. A path $\gamma$ is {\it rectifiable} if its length is finite.

Given a rectifiable path $\gamma\colon I\to \X$ such that $\ell_{d}(\gamma,0,s)$ is absolutely continuous with respect to $s$, we denote the \emph{length element} of $\gamma$ by 
\begin{equation}
\label{lengthelement}
\| D\gamma\| = \|D\gamma(s)\| = \frac{d}{ds} \ell(\gamma,0,s)\quad \mathrm{for\ a.e.}\ s\in I.
\end{equation}
Recall that an increasing absolutely continuous function is a.e. differentiable and can be recovered by integrating its derivative. Thus
\[
\ell (\gamma,0,t) = \int_0^t \|D\gamma\|\ ds=\int_{0}^{t}\|d\gamma\|,
\]
where the last integral can be interpreted as the Stieltjes integral with respect to integrator $\ell_{d}(\gamma,0,t)$, or equivalently, the Lebesgue integral, under the formal convention that
\begin{equation}\label{eq: deriv}
\| d\gamma\|=\|D\gamma\|\ ds,
\end{equation}
bearing \eqref{lengthelement} in mind. In this paper both interpretations for the integrals are useful. Note that for instance the parameterization with respect to the arc length is absolutely continuous. 
Obviously, any rectifiable path in a Banach space can be approximated uniformly by an absolutely continuous path, e.g., a broken line. If $\gamma$ is a path in a Banach space $\X$, we will denote its G\^ateaux derivative by 
\begin{equation}\label{eq: gateaux} 
D\gamma(t):=\lim_{h\to 0}\frac{\gamma(t+h)-\gamma(t)}{h},
\end{equation}
provided that it exists. Observe that if $\gamma$ is G\^ateaux differentiable at $t$, then $\| D\gamma(t)\|=\|(D\gamma(t))\|$. Here the left-hand side is as in \eqref{lengthelement}, whereas the right-hand side is the norm 
of a vector given by \eqref{eq: gateaux}. The latter is required in order to incorporate geometric considerations in path-length estimates.
We note that differentiation of Banach space valued functions can also be studied by means of the Bochner integral. This approach is effective especially in Banach spaces with the so-called Radon-Nikod\'ym property (RNP), 
which means that any rectifiable, absolutely continuous path starting from the origin can be recovered by Bochner integrating its G\^ateaux derivative. For basic information about these concepts we refer to \cite{Diestel} and \cite{DiestelUhl77}, 
see also \cite{Banach_space_book}. 

\subsection{Quasihyperbolic metric}

Let $\X$ be a Banach space with $\dim \X\ge 2$, and suppose that $\Omega \subsetneq \X$ is a domain.  For $x\in \Omega$, let $d(x)$ denote the distance $d(x,\partial \Omega)$. We define the \emph{quasihyperbolic length} of $\gamma$ by 
\[
\ell_k(\gamma):=\int_{I} \frac{\|d \gamma\|}{d(\gamma(t))}
\]
then the \emph{quasihyperbolic distance} of points $x,y\in \Omega$ is the number
\[
k(x,y):=k_\Omega(x,y):=\inf_\gamma \ell_k(\gamma)
\]
where the infimum is taken over all rectifiable arcs $\gamma$ joining $x$ and $y$ in $\Omega$.
Quasihyperbolic balls are
\begin{eqnarray*}
\U_k(x,r) &:=& \{ y\in  \Omega : k_\Omega(x,y) <r\},\\
\B_k(x,r) &:=& \{ y\in  \Omega : k_\Omega(x,y) \le r\}. 
\end{eqnarray*}

It is well known \cite[Lemma 1]{GehringOsgood79} that in the finite-dimensional case there is a quasihyperbolic geodesic between any two points. By \cite[Theorem 2.5]{Vaisala05}, for a reflexive Banach space $\X$ and a convex subdomain $\Omega\subsetneq \X$ there always exists a quasihyperbolic geodesic connecting $x,y\in \Omega$. One of the peculiarities of this topic is that it is not known whether this holds for general Banach spaces (see also \cite[Section 6]{Vaisala05}). It is easy to check that multiplication by a constant $C\neq 0$ is a quasihyperbolic isometry on $\Omega=\X\setminus \{0\}$. 

\subsection{Distance-ratio metric}

The quasihyperbolic distance is often difficult to compute in practice. For this reason, we consider another related quantity, the distance-ratio metric. This metric was originally introduced by Gehring and Palka in \cite{GehringPalka76}. 
We use a version that is due to Vuorinen \cite{Vuorinen85}. Let $\X$ be a Banach space with $\dim \X\ge 2$, and suppose that $\Omega \subsetneq \X$ is a domain. Write
\[
a \lor b := \max\{a,b\},\qquad a \land b := \min\{a,b\}.
\]
The {\it distance-ratio metric}, or {\it $j$-metric}, on $\Omega$ is defined by
\begin{equation}
\label{jdef}
j(x,y):=j_\Omega(x,y) := \log\bigg(1+\frac{\|x-y\|}{d(x)\land d(y)}\bigg), \qquad x,y\in \Omega.
\end{equation}
Again, the balls with respect to the $j$-metric are
\begin{eqnarray*}
\U_j(x,r) &:=& \{ y\in  \Omega : j_\Omega(x,y) <r\},\\
\B_j(x,r) &:=& \{ y\in  \Omega : j_\Omega(x,y) \le r\}. 
\end{eqnarray*}
It is well known that the norm metric, the quasihyperbolic metric and the distance-ratio metric define the same topology on $\Omega$. In fact, the $j$-metric is an inner metric of the quasihyperbolic metric.

\subsection{Geometric control of Banach spaces}

Next, we will recall for convenience two essential moduli related to the geometry of Banach spaces. The \emph{modulus of convexity} $\delta_{\X}(\epsilon),\ 0<\epsilon \leq 2,$ is defined by
\[\delta_{\X}(\epsilon):=\inf\{1-\|x+y\|/2:\ x,y\in \X,\ \|x\|=\|y\|=1,\ \|x-y\|=\epsilon\},\]
and the \emph{modulus of smoothness} $\rho_{\X}(\tau),\ t>0$ is defined by 
\[\rho_{\X}(\tau):=\sup\{(\|x+y\|+\|x-y\|)/2 -1,\ x,y\in \X,\ \|x\|=1,\ \|y\|=\tau\}.\]
The Banach space $\X$ is called \emph{uniformly convex} if $\delta_{\X}(\epsilon)>0$ for all $\epsilon>0$, and \emph{uniformly smooth}  if 
\[
\lim_{\tau\to 0^{+}}\frac{\rho_{\X}(\tau)}{\tau}=0.
\]
Moreover, a space $\X$ is uniformly convex (resp. uniformly smooth) of power type $p\in [1,\infty)$ if $\delta_{\X}(\epsilon)\geq K\epsilon^{p}$ (resp. $\rho_{\X}(\tau)\leq K\tau^p$) for some $K>0$. Note that the modulus $\delta_X$ measures the convexity of the unit ball. Uniform convexity of a general convex set can be defined analogously.


\section{Starlikeness of $j$-balls}

Next, we show that $j$-metric balls are starlike for radii $r\le \log 2$.

\begin{theorem}\label{thm: sl}
Let $\X$ be a Banach space, $\Omega\subsetneq\X$ a domain, and let $j$ be as in (\ref{jdef}). Then each $j$-ball  $\B_{j}(x_{0},r),\ x_{0}\in \Omega,$ is starlike for radii $r\leq \log 2$.
\end{theorem}
\begin{proof}
Let $x_{0},y\in \Omega$ such that $j(x_{0},y)\leq \log 2$. This is to say that 
\[
\frac{\|x_{0}-y\|}{d(x_{0})\wedge d(x)}\leq 1.
\]
By using simple calculations and the triangle inequality we get
\begin{eqnarray*}
j(x_{0},ty+(1-t)x_{0})&=&\log\left(1+\frac{\|x_{0}-(ty+(1-t)x_{0})\|}{d(x_{0})\wedge d(ty+(1-t)x_{0})}\right)\\
&\leq& \log\left(1+\frac{(1-t)\|x_{0}-y\|}{d(x_{0})\wedge (d(y)-t\|x_{0}-y\|)}\right)\leq \log 2,\\
\end{eqnarray*}
where we applied the fact $d(x_{0}), d(y)\geq \|x_{0}-y\|$ in the last inequality.
\end{proof}

\begin{proposition}\label{prop: intersection}
Let $\X$ be a Banach space and $\Omega\subset \X$ a domain with $\partial \Omega\neq \emptyset$. 
Then $\B_{j_{\Omega}}(x,r)=\bigcap_{z\in \X\setminus \Omega}\B_{j_{\X\setminus \{z\}}}(x,r)$. Moreover, if $\X$ is reflexive and 
$\Omega$ is weakly open, then 
\[
\U_{j_{\Omega}}(x,r)=\bigcap_{z\in \X\setminus \Omega}\U_{j_{\X\setminus \{z\}}}(x,r).
\]
\end{proposition}
\begin{proof}
Denote by $C$ the norm closed set $\X\setminus \Omega$. First, note that $\X\setminus C \subset \X\setminus \{z\}$ and that $j_{\X\setminus \{z\}}(x,y)\leq j_{\Omega}(x,y)$ for each $z\in C$ and $x,y\in \Omega$.
Thus  $\B_{j_{\Omega}}(x,r)\subset \bigcap_{z\in C}\B_{j_{\X\setminus \{z\}}}(x,r)$ and $\U_{j_{\Omega}}(x,r)\subset \bigcap_{z\in C}\U_{j_{\X\setminus \{z\}}}(x,r)$.
Pick $y\in \Omega$ such that 
\[
j(x,y)=\log \bigg(1+\frac{\|x-y\|}{d(x)\wedge d(y)}\bigg)>r.
\]
Then there is $z\in C$ such that 
\[
\log \bigg(1+\frac{\|x-y\|}{\|x-z\|\wedge \|y-z\|}\bigg)>r.
\] 
This means that $y\notin \bigcap_{z\in C}\B_{j_{\X\setminus \{z\}}}(x,r)$, and so we have the first part of the statement.

Now, assume that $\X$ is reflexive and $\Omega$ is weakly open. Pick $y\in \Omega$ with $j(x,y)=r_{0}\geq r$. Let $v\in \{x,y\}$ and $s_{0}\in \R$ be such that 
\[
r_{0}=\log \bigg(1+\frac{\|x-y\|}{d(v)}\bigg)=\log \bigg(1+\frac{\|x-y\|}{s_{0}}\bigg).
\] 
Note that $C$ is weakly closed and thus by James' well-known characterization of reflexivity of Banach spaces (see e.g. \cite{Banach_space_book}), we get that $\B_{\|\cdot\|}(v,s_{0}+1)\cap C$ is weakly compact. 
Thus
\[
\bigcap_{\epsilon>0} \B_{\|\cdot\|}(v,s_{0}+\epsilon)\cap C\neq \emptyset,
\]
so let us select a point $z$ from this set. Note that $\|v-z\|=s_{0}$, since $d(v)=s_{0}$.
This means that
\[
j_{\X\setminus \{z\}}(x,y)\geq\log \bigg(1+\frac{\|x-y\|}{\|v-z\|}\bigg)=r_{0}\geq r.
\]
Consequently, 
$\U_{j_{\Omega}}(x,r)\subset \bigcap_{z\in C}\U_{j_{\X\setminus \{z\}}}(x,r)$.
\end{proof}




\begin{remark}\label{remark: unif_conformal}
The quasihyperbolic metric on $\X\setminus \{0\}$ is conformal in the following sense:
for each $C>1$ there is $r>0$ such that 
\[
C^{-1}k(x,y)\leq \frac{\|x-y\|}{\|x\|}\leq Ck(x,y)
\]
for $k(x,y)<r$. The same is true for the distance ratio metric. Note that we did not assume anything about the geometry of $\X$. The proof follows the arguments in \cite[p. 35]{Vuorinen88}, and is left to the reader. 
\end{remark}

\begin{remark}\label{remark_Hilbert}
Kl\'{e}n's main results in \cite{Klen08a} and \cite{Klen08b} involving $\R^{n}$ can be adapted to general (finite-dimensional, separable, non-separable, real or complex) Hilbert spaces $\H$.
This is due to the fact that the core of the arguments is, roughly speaking, based on calculations in $\R^{2}$ and then these observations extend to $\R^{n}$ by elegant reasoning. 
Essentially the same extension carries further to Hilbert spaces.
\end{remark}

\section{Convexity of quasihyperbolic and $j$-balls on convex domains}

In this section, we study convexity of quasihyperbolic and $j$-metric balls. We present a generalization of a result of Martio and V\"ais\"al\"a \cite[2.13]{MartioVaisala11}.

\begin{theorem}\label{thm: cd} 
Let $\X$ be a Banach space and $\Omega\subsetneq\X$ a convex domain. Then all quasihyperbolic balls and $j$-balls on $\Omega$ are convex. Moreover, if $\Omega$ is uniformly convex, or if $\X$ is strictly convex and has the RNP, then
these balls are strictly convex.
\end{theorem}

\begin{fact}\label{fact abcd}
Let $a,b,c,d>0$ be constants such that $a/c = b/d$. Then 
\[\frac{ta+(1-t)b}{tc+(1-t)d}=\frac{a}{c}\quad \mathrm{for}\ t\in [0,1].\]
\end{fact}


\begin{proof}[Proof of Theorem \ref{thm: cd}]
We will prove the case with the quasihyperbolic metric, which is more complicated. Fix $x\in \Omega$ and $r>0$. Let $y,z\in \B_{k}(x,r)$. Our aim is to verify that 
$sy+(1-s)z\in \B_{k}(x,r)$ for $s\in [0,1]$. 
Thus, we may assume that $k(x,y)=k(x,z)=r$ in the first place. By using suitable translations, we may assume that $x=0$ as well. It suffices to show that 
\begin{equation}\label{eq: kyz}
k(x,sy+(1-s)z)\leq r,\quad \mathrm{for}\ s\in [0,1].
\end{equation}
We use the following short-hand notation
\[
\ell_k(\gamma,t_{1},t_{2})=\int_{t_{1}}^{t_{2}}\frac{\|d\gamma(t)\|}{d(\gamma(t))},
\]
where $\gamma\colon I\to \Omega$ is a rectifiable path and $0\leq t_{1}\leq t_{2}\leq 1$. We will also write $\ell_k(\gamma)$ instead of $\ell_k(\gamma,0,1)$.

Let $\epsilon>0$ and let $\gamma_{0},\gamma_{1}\colon I\to \Omega$ be rectifiable paths such that $\gamma_{0}(0)=\gamma_{1}(0)=0$, $\gamma_{0}(1)=z,\ \gamma_{1}(1)=y$, 
$\ell_k(\gamma_{0}) \leq r+\epsilon$ and $\ell_k(\gamma_{1})\leq r+\epsilon$.
We may assume by symmetry that $\ell_k(\gamma_{0})\leq \ell_k(\gamma_{1})$. Moreover, by suitably modifying and then re-parameterizing $\gamma_{0}$, we may assume that 
$\ell_k(\gamma_{0})=\ell_k(\gamma_{1})$ and $\ell_k(\gamma_{0},0,t)=\ell_k(\gamma_{1},0,t)=t\ell_k(\gamma_{0})$ for $t\in [0,1]$. Thus we have that 
\begin{equation}\label{eq: ae}
\frac{\|D\gamma_{0}(t)\|}{d(\gamma_{0}(t))}=\frac{\|D\gamma_{1}(t)\|}{d(\gamma_{1}(t))}=\ell_k(\gamma_{0})=\ell_k(\gamma_{1})\quad \mathrm{for\ a.e.}\ t\in [0,1].
\end{equation}
Observe that the above numerators need not be continuous, so that these terms do not coincide, at least a priori, for every $t$.

\begin{figure}
\begin{center}
\includegraphics[width=4.5cm]{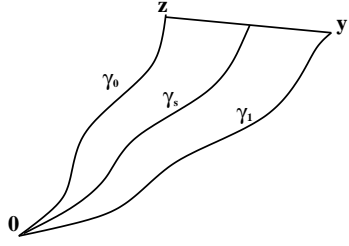}
\end{center}
\caption{The average path $\gamma_s$.}\label{avgfig}
\end{figure}

Define an average path $\gamma_{s}$ (see Figure \ref{avgfig}) for $s\in [0,1]$ by $\gamma_{s}=s\gamma_{1}+(1-s)\gamma_{0}$. Clearly, $\gamma(0)_{s}=0$ and $\gamma_{s}(1)=sy+(1-s)z$ for $s\in [0,1]$. 
We claim that 
\begin{equation}\label{eq: int1}
\ell_k(\gamma_{s}) \leq s\ell_k(\gamma_{1}) + (1-s)\ell_k(\gamma_{0})=\ell_k(\gamma_{0})=\ell_k(\gamma_{1}).
\end{equation}
Because $\epsilon$ was arbitrary, this estimate yields \eqref{eq: kyz}, which provides the required result. 

To obtain the estimate \eqref{eq: int1}, observe that the inequality
\begin{equation}\label{eq: dgamma}
\|D\gamma_{s}\|\leq s\|D\gamma_{1}\|+(1-s)\|D\gamma_{0}\|,\quad \mathrm{for}\ s\in I\
\end{equation}
holds pointwise for a.e. $t\in I$ in the sense of \eqref{lengthelement}. Indeed, here we recall the definition of the norm length $\ell$ and apply the triangle inequality. 
Given $v,u\in \Omega$, it holds that 
\[\U_{\|\cdot\|}(v,d(v))\cup \U_{\|\cdot\|}(u,d(u))\subset \Omega,\] 
and by the convexity of $\Omega$, it holds that 
\[\{sa+(1-s)b:\ a\in \U_{\|\cdot\|}(v,d(v)),\  b\in \U_{\|\cdot\|}(u,d(u)),\ s\in [0,1]\}\subset\Omega.\] 
Moreover, the above set contains $\U_{\|\cdot\|}(sv+(1-s)u,sd(v)+(1-s)d(u))$, see Figure \ref{uvballfig}. See also \cite[Lemma 3.5]{Vaisala05}.

\begin{figure}
\begin{center}
\includegraphics[width=6cm]{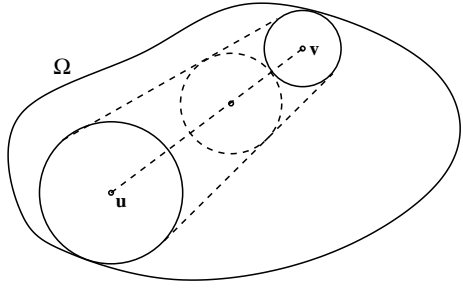}
\end{center}
\caption{The ball $\U_{\|\cdot\|}(sv+(1-s)u,sd(v)+(1-s)d(u))$.}\label{uvballfig}
\end{figure}

This means that
\begin{equation}\label{eq: dyz}
d(su+(1-s)v)\geq sd(u)+(1-s)d(v).
\end{equation}
Now, by combining \eqref{eq: dgamma}, \eqref{eq: dyz}, \eqref{eq: ae} and Fact \ref{fact abcd} we obtain
\[
\ell_k(\gamma_{s})=\int_{0}^{1}\frac{\|d \gamma_{s}(t)\|}{d(\gamma_{s}(t))}\leq \int_{0}^{1}\frac{s\|d \gamma_{1}(t)\|+(1-s)\|d \gamma_{0}(t)\|}{d(\gamma_{s}(t))}
\]
\[
\leq \int_{0}^{1}\frac{s\|d \gamma_{1}(t)\|+(1-s)\|d \gamma_{0}(t)\|}{sd(\gamma_{1}(t))+(1-s)d(\gamma_{0}(t))}=\int_{0}^{1}\ell_k(\gamma_{0}) \,dt=\ell_k(\gamma_{0}).
\]
This completes the proof for the first part of the statement. 

In the latter part, suppose that $\gamma_{0}\neq\gamma_{1}$. Then $\gamma_{s}(t),\ 0<s<1$: 
\begin{itemize}
\item{Satisfies \eqref{eq: dyz} strictly for a set of values of $t$ having positive measure if $\Omega$ is uniformly convex.}
\item{Satisfies \eqref{eq: dgamma} strictly for a set of values of $t$ having positive measure if $\X$ is strictly convex and has the RNP.} 
\end{itemize}
The strict convexity of the quasihyperbolic balls follows. The proof for the $j$-metric is similar.
\end{proof}

\subsection{Starlike domains}

Next we show that if $\Omega$ is starlike with respect to $x_0\in \Omega$, then all quasihyperbolic and $j$-metric balls centered at $x_0$ are starlike as well.

\begin{theorem}\label{thm_starlike_starlike}
Let $\X$ be Banach space, $x_{0}\in \X$ and let $\Omega\subset \X$ be a domain which is starlike with respect to $x_{0}$. Then all balls $\B_{j}(x_{0},r)$ and $\B_{k}(x_{0},r)$ of $\Omega$ are starlike.
\end{theorem}
\begin{proof}
We will only consider the case with the quasihyperbolic metric, since the other case is similar but easier. Fix a path $\gamma\colon I\to \Omega$ with $\gamma(0)=x_{0}$ and $\ell_{k}(\gamma)<\infty$.
Our aim is to show that the path $\gamma_{s}$ defined by $\gamma_{s}(t)=s\gamma(t)+(1-s)x_{0}$ satisfies $\ell_{k}(\gamma_{s})\leq \ell_{k}(\gamma)$ for every $s\in I$. This suffices for the claim, as it follows easily that 
the $k$-ball centered at $x_{0}$ will be starlike. 

Similarly as in the proof of Theorem \ref{thm: cd} we observe that 
\[\{sy+(1-s)x_{0}:\ y\in \B_{\|\cdot\|}(x,r)\}=\B_{\|\cdot\|}(sx+(1-s)x_{0},sr)\]
for $x\in \X$, $r>0$ and by using the starlikeness of $\Omega$ we get 
\begin{equation}\label{eq_sd}
d(sx+(1-s)x_{0})\geq sd(x)\quad \mathrm{for}\ x\in \Omega.
\end{equation}
On the other hand, it is clear that 
\begin{equation}\label{eq_deriv_starlike}
\|D\gamma_{s}\|= s\|D\gamma\|\quad \mathrm{for\ a.e.}\ t\in I
\end{equation}
in the sense of \eqref{lengthelement}. Thus, it follows that 
\[\ell_{k}(\gamma_{s})=\int_{I}\frac{\|d\gamma_{s}\|}{d(\gamma_{s})}=\int_{I}\frac{s\|d\gamma\|}{d(\gamma_{s})}\leq \int_{I}\frac{s\|d\gamma\|}{sd(\gamma)}=\ell_{k}(\gamma).\]
\end{proof}

\section{Examples}

It is a natural question, if for each Banach space $\X$ and domain $\Omega\subset \X$ there is a critical radius 
$R_{\Omega}>0$ such that $j$ and $k$ balls on $\Omega$ with radius at most $R_{\Omega}$ are convex. The next 
example shows that this is not the case for $j$-balls. This example also appears to provide evidence suggesting that the similar holds for $k$-balls.

\begin{figure}[h]
\begin{center}
\includegraphics[width=6cm]{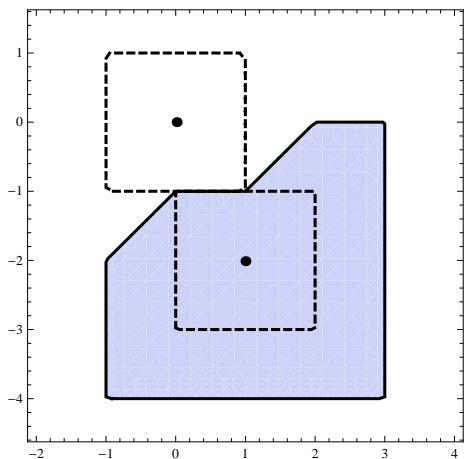}
\end{center}
\caption{There is no critical radius $R>0$ such that the ball $\B_j(x,r)$ is convex for all $x\in \Omega = \ell^{\infty}(2)\setminus\{0\}$ and $0<r<R$.}\label{nonconvexballfig}
\end{figure}

This observation can be used further to obtain that there exists a reflexive, strictly convex and smooth space $\X$
such that the punctured space $\Omega=\X\setminus \{0\}$ admits no critical radius of convexity for 
$j$-balls. Namely, this kind of space $\X$ is obtained by putting
\[\X=\ell^{2}(2)\oplus_{2} \ell^{3}(2)\oplus_{2} \ell^{4}(2) \oplus_{2} \ldots\]
This is a standard example of a Banach space, which is reflexive but not uniformly convex and not uniformly smooth.

In \cite[2.14]{MartioVaisala11} Martio and V\"{a}is\"{a}l\"{a} asked whether the quasihyperbolic balls of convex domains of uniformly convex Banach spaces are quasihyperbolically convex. 
More precisely, given two points $a$ and $b$ of the quasihyperbolic ball $B\subset \Omega$, does there exist a geodesic $\gamma$ joining $a$ and $b$, which is contained in the ball $B$. 
Here the domain $\Omega$ was assumed to be convex and the length of the geodesic is measured with respect to the quasihyperbolic metric. It turns out the the answer is negative, as the following counterexample shows.
\begin{example}
Let $\Omega=\{(x,y)\in \R^{2}:\ y<0\}$ and we will first consider $\Omega$ as a subset of $\ell^{\infty}(2)=(\R^{2},\|\cdot\|_{\infty})$. Let $x=(0,-1)$, 
\[
r=\ln(2)=\int_{1}^{2}t^{-1}\ dt.
\] 
We will study the ball $\B_{k}(x,r)$. Put $a=(-1,-2)$, $b=(1,-2)$ and observe that $\{ta+(1-t)b:\ t\in [0,1]\}$ is included in $\partial\B_{k}(x,r)$. An intuition, which helps in computing the quasihyperbolic lengths of paths, is that one can move to the directions 
$(-1,-1)$, $(0,-1)$ and $(1,-1)$ at the same cost because of the choice of the norm. Note that $z_{2}\geq -2$ for any $(z_{1},z_{2})\in \B_{k}(x,r)$.

Now, an easy computation shows that any path $\gamma\subset \B_{k}(x,r)$, which joins $a$ and $b$ must have quasihyperbolic length at least
\[
\int_{-1}^{1}\frac{1}{2}\ dt =1. 
\]
However, the broken line $\gamma_{0}$ connecting $a,b$ through the point $c=(0,-3)$ has length 
\[
2\int_{0}^{1}\frac{1}{3-t}\ dt=\ln\Big(\frac{9}{4}\Big)<1,
\]
see Figure \ref{counterexfig}. The existence of geodesics is clear in this choice of space. Thus $\B_{k}(x,r)$ is not quasiconvex.

\begin{figure}
\begin{center}
\includegraphics[width=7cm]{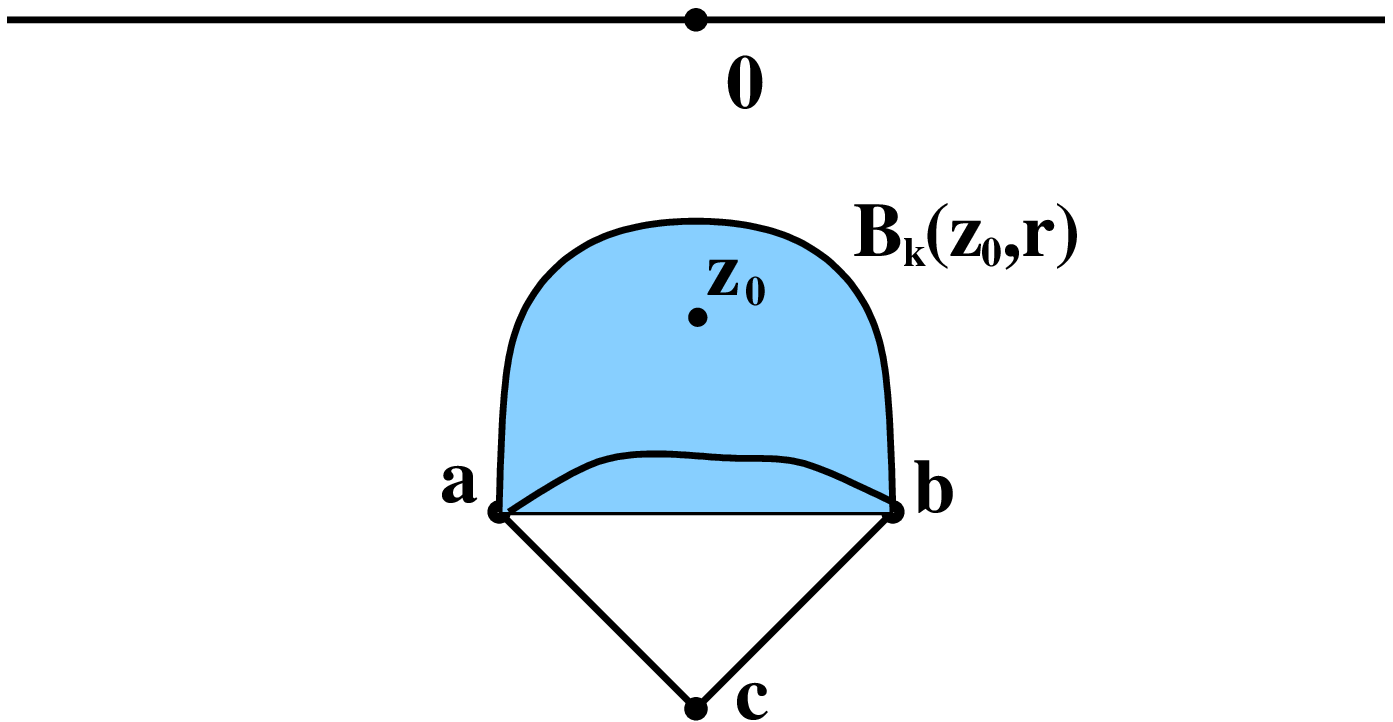}
\end{center}
\caption{The path $\gamma_0$ consists of line segments $[a,c]$ and $[c,b]$.}\label{counterexfig}
\end{figure}

This example does not change considerably if one considers the domain $\Omega=(-6,6)\times (0,6)$ instead. Observe that the space $\ell^{\infty}(2)$ is certainly not uniformly convex, see Figure \ref{nonconvexballfig}. However, because the quasihyperbolic metric depends continuously on the selection of the norm, we could apply the space $\ell^{p}(2)$ for large $p<\infty$ in place of $\ell^{\infty}(2)$ to produce similar examples, in which case we are dealing with uniformly convex spaces.
\end{example}

\section{Discussion: Convexity of quasihyperbolic balls in a punctured Banach space}

In the the work of Kl\'en \cite{Klen08a} critical radii are provided for the convexity of quasihyperbolic and $j$-balls on punctured $\R^n$.
Again, it is a natural question whether the existence of such radii can be established, mutatis mutandis, in the Banach space setting.
We already noted in Remark \ref{remark_Hilbert} that the same approach is applicable in general Hilbert spaces. On the other hand, as pointed out in the Examples section, 
there is a concrete reflexive, strictly convex and smooth Banach space such that when punctured, it has non-convex $j$-balls of arbitrary small $j$-radius. 
Presumably the similar holds for $k$-balls on the same space.

In order to check the convexity of a $k$-ball in a punctured space it would seem natural to exploit an averaging argument similar to the proof of Theorem \ref{thm: cd}. We have the following partial result which comes very close to providing such a device.

\begin{theorem}\label{lm: QHlemma}
Let $\X$ be a Banach space, which is uniformly convex and uniformly smooth, both moduli being of power type $2$. We consider the quasihyperbolic metric $k$ on $\Omega=\X\setminus \{0\}$. Then there exists $R>0$ as follows.
Assume that $\gamma_{1},\gamma_{2}\colon [0,t_{2}]\to \Omega$ are rectifiable paths satisfying the following conditions:
\begin{enumerate}
\item[(i)]{$\gamma_{1},\gamma_{2}$ and $(\gamma_{1}+\gamma_{2})/2$ are contained in $\B_{\|\cdot \|}(0,2)\setminus \B_{\|\cdot \|}(0,1)$},
\item[(ii)]{$\gamma_{1}(0)=\gamma_{2}(0)$},
\item[(iii)]{$\ell_{k}(\gamma_{1})\vee \ell_{k}(\gamma_{2})\leq R$} 
\item[(iv)]{$\ell_{\|\cdot\|}(\gamma_{1})=t_{1}\leq t_{2}=\ell_{\|\cdot\|}(\gamma_{2})$}
\item[(v)]{The paths are parameterized with respect to $\ell_{\|\cdot\|}$, except that\\ $\gamma_{1}(t)=\gamma_{1}(t_{1})$ for $t\in [t_{1},t_{2}]$.}
\end{enumerate}
Then the following estimate holds: 
\[
\frac{\ell_{k}(\gamma_{1})+\ell_{k}(\gamma_{2})}{2}+\int_{t_{1}}^{t_{2}}\frac{\|d\gamma_{2}\|}{2d(\gamma_{2})}
\geq \ell_{k}\left(\frac{\gamma_{1}+\gamma_{2}}{2}\right)+\int_{0}^{t_{1}}\frac{\delta_{\X}(\|D(\gamma_{1}-\gamma_{2})\|)}{\|\gamma_{1}\|+\|\gamma_{2}\|}\ ds.
\] 
\end{theorem}

Before giving the proof we will make some comments. If it were the case that $t_{1}=t_{2}$, then this auxiliary fact would provide the required inequlity, similar as applied in the proof of Theorem \ref{thm: cd}, to 
prove the existence of a critical radius of convexity of $k$-balls in a puncture Banach space.

Note that, because of the term
\[
\int_{t_{1}}^{t_{2}}\frac{\|d\gamma_{2}\|}{2d(\gamma_{2})},
\]
we do not obtain a critical radius for convexity in punctured Banach spaces. However, the above result yields that if $\gamma_{1}$ and $\gamma_{2}$ have both coinciding norm-lengths and $k$-lengths, then either the point-wise average path has strictly smaller $k$-length or 
$D(\gamma_{1}-\gamma_{2})=0$ a.e. The latter means that $\gamma_{1}=\gamma_{2}$, in fact. Thus we obtain partial information related to the \emph{strict} convexity of the $k$-balls and the \emph{uniqueness} of geodesics. In particular, if $t_1=t_2$ then either
\[
\frac{\ell_{k}(\gamma_{1})+\ell_{k}(\gamma_{2})}{2} > \ell_{k}\left(\frac{\gamma_{1}+\gamma_{2}}{2}\right),
\textrm{ or } \gamma_1=\gamma_2.
\]
 
To comment on the assumptions, the postulates $(i)$-$(v)$ are not restrictive for the purposes under discussion. Namely, one may apply suitable scalings which are isometries with respect to the $k$-metric.
Any Hilbert space has the best possible power types of uniform convexity and uniform smoothness, namely $p=2$, and, in fact, the optimal modulus functions. It is known that any Banach space has the uniform convexity power type at least $2$ 
and the uniform smoothness power type at most $2$. Our method in the proof of Theorem \ref{lm: QHlemma} requires that the asymptotics of the moduli should be essentially the same, and this is why we assumed that the power types of the moduli should coincide, 
i.e. $p=2$ for both accounts. It is perhaps worthwhile to pay close attention to how Lemma \ref{lm: Holder} is applied at the end of the proof. We note that any Banach space with the coinciding power types of the moduli must be linearly homeomorphic to a Hilbert space. 

\begin{lemma}\label{lm: Holder}
Let $f\in L^{p},\ 1\leq p<\infty,$ such that $f>0$ a.e. and let $F(t)=\int_{0}^{t}f(s)\ ds$, $0\leq t\leq 1$. Then 
\[\frac{\int_{0}^{t}F(s)^{p}\ ds}{\int_{0}^{t}f(s)^{p}\ ds}\leq t^{p} \quad \mathrm{for}\ 0\leq t\leq 1.\]
\end{lemma}
\begin{proof}
We will apply the well-known fact that the expectation operator on $L^{p}([0,t])$ is contractive, which is easiest to see by writing it like $1\otimes 1/t$. Then we have
\[
\frac{\int_{0}^{t}F(s)^{p}\ ds}{\int_{0}^{t}f(s)^{p}\ ds}\leq \frac{tF(t)^{p}}{\int_{0}^{t}(F(t)/t)^p\ ds}=\frac{tF(t)^{p}}{F(t)^{p}/t^{p-1}}=t^{p}.
\]
\end{proof}
In the above lemma it is essential that the exponents appearing in the numerator and the denominator are the same. This can be seen by multiplying $f$ with suitable positive constants, as $F$ depends linearly on $f$.

\begin{proof}[Proof of Theorem \ref{lm: QHlemma}]
We note that the assumption about the parameterization yields that 
\[
\|D\gamma_{1}(t)\| = \|D \gamma_{2}(t)\| = 1\quad \mathrm{for\ a.e.}\ t\in [0,t_{1}].
\]
Recall that we denote the G\^ateaux derivative of a path $\gamma$ by $D\gamma$. Since $\X$ has the RNP, being a reflexive space, it follows that each reasonably parameterized path of finite quasihyperbolic length 
is differentiable almost everywhere and can be recovered from its derivative by Bochner integration. 

By using assumption (i) we observe that 
\begin{eqnarray*}
\ell_{k}\left(\frac{\gamma_{1}+\gamma_{2}}{2},t_{1},t_{2}\right)&=&\int_{t_{1}}^{t_{2}}\frac{\|D(\frac{\gamma_{1}+\gamma_{2}}{2})\|}{\|\frac{\gamma_{1}+\gamma_{2}}{2}\|}\ ds=\int_{t_{1}}^{t_{2}}\frac{\|D \gamma_{2}\|}{2\|\frac{\gamma_{1}(t_{1})+\gamma_{2}}{2}\|}\ ds\\
&\leq& \int_{t_{1}}^{t_{2}}\frac{\|D \gamma_{2}\|}{2}\ ds \leq \int_{t_{1}}^{t_{2}}\frac{\|D \gamma_{2}\|}{\|\gamma_{2}\|}\ ds=\ell_{k}(\gamma_{2},t_{1},t_{2}).
\end{eqnarray*}

Thus our task reduces to verifying that
\[
\frac{\ell_{k}(\gamma_{1},0,t_{1})+\ell_{k}(\gamma_{2},0,t_{1})}{2}\geq \ell_{k}\left(\frac{\gamma_{1}+\gamma_{2}}{2},0,t_{1}\right)+\int_{0}^{t_{1}}\frac{\delta_{\X}(\|D(\gamma_{1}-\gamma_{2})\|)}{\|\gamma_{1}\|+\|\gamma_{2}\|}\ ds.
\] 

Without loss of generality we may assume, possibly by re-defining the paths, that $\|D(\gamma_{1}-\gamma_{2})(t)\|$ is not zero in any open neighborhood of $0$.

Let us evaluate by using the convexity of the mapping $t\mapsto t^{-1}$ and the moduli of smoothness and convexity in the following manner:
\begin{eqnarray*}
& & \frac{1}{2}\bigg(\frac{\|D\gamma_{1}\|}{\|\gamma_{1}\|} + \frac{\|D \gamma_{2}\|}{\|\gamma_{2}\|}\bigg)=\frac{1}{2}\left(\frac{1}{\|\gamma_{1}\|}+\frac{1}{\|\gamma_{2}\|}\right)\\
&\geq &\frac{2}{\|\gamma_{1}\|+\|\gamma_{2}\|}
\geq \frac{\|D(\gamma_{1}+\gamma_{2})\|}{\|\gamma_{1}\|+\|\gamma_{2}\|}+\frac{2\delta_{\X}(\|D(\gamma_{1}-\gamma_{2})\|)}{\|\gamma_{1}\|+\|\gamma_{2}\|}\\
&\geq &\frac{\|D(\gamma_{1}+\gamma_{2})\|}{\|\gamma_{1}+\gamma_{2}\|\big(1+2\rho_{\X}\big(\frac{\|\gamma_{1}-\gamma_{2}\|}{2\|\gamma_{1}+\gamma_{2}\|}\big)\big)} + \frac{2\delta_{\X}(\|D(\gamma_{1}-\gamma_{2})\|)}{\|\gamma_{1}\|+\|\gamma_{2}\|}.\\
\end{eqnarray*}

We aim to verify that there exists $R>0$ such that
\begin{eqnarray*}
& & \int_{0}^{t}\frac{\|D(\gamma_{1}+\gamma_{2})\|}{\|\gamma_{1}+\gamma_{2}\|\big(1+2\rho_{\X}\big(\frac{\|\gamma_{1}-\gamma_{2}\|}{ 2\|\gamma_{1}+\gamma_{2}\|}\big)\big)}+ \frac{2\delta_{\X}(\|D(\gamma_{1}-\gamma_{2})\|)}{\|\gamma_{1}\|+\|\gamma_{2}\|}\ ds\\
&\geq & \int_{0}^{t}\frac{\|D(\gamma_{1}+\gamma_{2})\|}{\|\gamma_{1}+\gamma_{2}\|}+\frac{\delta_{\X}(\|D(\gamma_{1}-\gamma_{2})\|)}{\|\gamma_{1}\|+\|\gamma_{2}\|}\ ds\\
\end{eqnarray*}
for all $0\leq t\leq R$. Recall that $1\leq \|\gamma_{1}+\gamma_{2}\|\leq 4$ by the assumptions. 
Let us analyze the terms of the above inequality:
\begin{eqnarray*}
& &  \int_{0}^{t}\frac{\|D(\gamma_{1}+\gamma_{2})\|}{\|\gamma_{1}+\gamma_{2}\|} ds - \int_{0}^{t}\frac{\|D(\gamma_{1}+\gamma_{2})\|}{\|\gamma_{1}+\gamma_{2}\|\big(1+2\rho_{\X}\big(\frac{\|\gamma_{1}-\gamma_{2}\|}{ 2\|\gamma_{1}+\gamma_{2}\|}\big)\big)}\ ds\\
&=&  \int_{0}^{t}\frac{\|D(\gamma_{1}+\gamma_{2})\|}{\|\gamma_{1}+\gamma_{2}\|} \left(1- \frac{1}{\big(1+2\rho_{\X}\big(\frac{\|\gamma_{1}-\gamma_{2}\|}{ 2\|\gamma_{1}+\gamma_{2}\|}\big)\big)}\right)\ ds\\
&\leq& \int_{0}^{t} \left(1-\frac{1}{(1+2\rho_{\X}(\|\gamma_{1}-\gamma_{2}\| / 8))}\right) ds\leq \int_{0}^{t} 2\rho_{\X}(\|\gamma_{1}-\gamma_{2}\| / 8)\ ds\\
\end{eqnarray*}
and
\begin{eqnarray*}
\int_{0}^{t}\delta_{\X}(\|D(\gamma_{1}-\gamma_{2})\|)/ 8\ ds\leq \int_{0}^{t}\frac{\delta_{\X}(\|D(\gamma_{1}-\gamma_{2})\|)}{\|\gamma_{1}\|+\|\gamma_{2}\|}\ ds.\\
\end{eqnarray*}
To justify the existence of the claimed constant $R>0$ it suffices to check that 
\begin{equation}\label{eq: zero} 
\frac{\int_{0}^{t} 2\rho_{\X}(\|\gamma_{1}-\gamma_{2}\| / 8)\ ds}{\int_{0}^{t}\delta_{\X}(\|D(\gamma_{1}-\gamma_{2})\|)/ 8\ ds}\longrightarrow 0
\end{equation}
uniformly, regardless of the selection of paths, as $t\to 0$. 

Define  $f(s)=\|D(\gamma_{1}-\gamma_{2})(s)\|$ for a.e. $s\in [0,r]$ and put
\[
F(t)=\int_{0}^{t}f(s)\ ds\geq \|\gamma_{1}(t)-\gamma_{2}(t)\|.
\]
Recall that $\rho_{\X}(\tau)\leq K \tau^{2}$ and $\delta_{\X}(\epsilon)\geq M \epsilon^{2}$.
Then the above ratio in \eqref{eq: zero} can be evaluated from above by
\begin{equation}\label{eq: ratio2}
\frac{\int_{0}^{t}2\rho(F(s)/8)\ ds}{\int_{0}^{t}\delta_{\X}(f(s))/ 8\ ds}\leq \frac{1}{4}M^{-1}K \frac{\int_{0}^{t} F(s)^{2}\ ds}{\int_{0}^{t}f(s)^{2}\ ds}\leq \frac{1}{4}M^{-1}K t^{2}.
\end{equation}
Above we applied Lemma \ref{lm: Holder}, and we note that the right-hand side tends to $0$ as $t\to 0$, independently of the choice of $f$. Thus we have the claim.
\end{proof}

\subsection*{Acknowledgments} 
We thank J.~V\"ais\"al\"a and M.~Vuorinen for their helpful comments on this paper.

\end{document}